\documentclass[a4paper,11pt]{article}

\usepackage{amssymb,amsmath,amsthm,amsopn,bm}
\usepackage[makeroom]{cancel}
\usepackage{mathtools}
\usepackage{graphicx}

\DeclareGraphicsExtensions{.pdf,.png,.eps}
\graphicspath{{figs/}{Figures/}}

\usepackage[hidelinks]{hyperref}
\usepackage[capitalize]{cleveref}
\usepackage{enumitem}
\usepackage{pifont}
\usepackage[dvipsnames]{xcolor}

\newcommand{\Host}{\ensuremath{H}}
\newcommand{\Guest}{\ensuremath{G}}

\newcommand{\Guestclass}{\ensuremath{\mathfrak{F}}}

\newcommand{\CBipartites}{{\ensuremath{\mathfrak{C}\mathfrak{B}}}}
\newcommand{\Differences}{\ensuremath{\mathfrak{D}}}

\newcommand{\cn}[3]{\ensuremath{\operatorname{c}_{\ensuremath{#1}}^{#2}(#3)}}
\makeatletter
\def\moverlay{\mathpalette\mov@rlay}
\def\mov@rlay#1#2{\leavevmode\vtop{%
   \baselineskip\z@skip \lineskiplimit-\maxdimen
   \ialign{\hfil$\m@th#1##$\hfil\cr#2\crcr}}}
\newcommand{\charfusion}[3][\mathord]{
    #1{\ifx#1\mathop\vphantom{#2}\fi
        \mathpalette\mov@rlay{#2\cr#3}
      }
    \ifx#1\mathop\expandafter\displaylimits\fi}
\makeatother
\newcommand{\cupdot}{\charfusion[\mathbin]{\cup}{\cdot}}

\newcommand{\disver}{\ensuremath{\mathop{\cupdot}}}

\DeclareMathOperator{\ldim}{\operatorname{ldim}} 
\DeclareMathOperator{\fdim}{\operatorname{fdim}} 

\newtheorem{theorem}{Theorem}
\crefname{theorem}{Theorem}{Theorems}
\newtheorem{lemma}[theorem]{Lemma}
\crefname{lemma}{Lemma}{Lemmas}
 \newtheorem{proposition}[theorem]{Proposition}
\crefname{proposition}{Proposition}{Propositions}
\newtheorem{corollary}[theorem]{Corollary}
\crefname{corollary}{Corollary}{Corollaries}
\newtheorem{problem}{Problem}

\crefname{figure}{Figure}{Figures}

\def\ovl{\overline}

\begin{document}

\title{On Covering Numbers, Young Diagrams, and the Local Dimension of Posets}
\author{
 G\'abor Dam\'asdi\\{\small ELTE Budapest} \\ {\small\tt damasdigabor@caesar.elte.hu} \and
 Stefan Felsner\footnote{Partially supported by DFG grant FE-340/13-1}\\
 {\small TU Berlin} \\ {\small\tt felsner@math.tu-berlin.de}\and
 Ant\'onio Gir\~ao \\{\small School of Mathematics} \\ {\small University of Birmingham} \\ {\small\tt giraoa@bham.ac.uk} \and
 Bal\'azs Keszegh\footnote{Supported by the Lend\"ulet program of the Hungarian Academy of Sciences (MTA), under the grant LP2017-19/2017 and by NKFIH grant K116769.} \\{\small Alfr\'ed R\'enyi Institute of Mathematics, Budapest} \\{\small MTA-ELTE Lend\"ulet Combinatorial Geometry Research Group} \\{\small\tt keszegh@renyi.hu} \and
 David Lewis \\{\small Department of Mathematical Sciences} \\ {\small University of Memphis} \\ {\small\tt davidcharleslewis@outlook.com} \and
 D\'aniel T. Nagy\footnote{Supported by NKFIH grants K132696 and FK132060.} \\{\small Alfr\'ed R\'enyi Institute of Mathematics} \\ {\small\tt nagydani@renyi.hu}  \and
 Torsten Ueckerdt\\ {\small Institute of Theoretical Informatics} \\ {\small Karlsruhe Institute of Technology}\\ {\small\tt torsten.ueckerdt@kit.edu}
}

\maketitle

\begin{abstract}
  \noindent
  We study covering numbers and local covering numbers with respect to
  difference graphs and complete bipartite graphs. In particular we
  show that in every cover of a Young diagram with $\binom{2k}{k}$
  steps with generalized rectangles there is a row or a column in the
  diagram that is used by at least $k+1$ rectangles, and prove that this is
  best-possible. This answers two questions
  by Kim, Martin, Masa{\v{r}}{\'\i}k, Shull, Smith, Uzzell, and
  Wang~\cite{KMMSSUW18}, namely:
  \begin{enumerate}
  \item What is the local complete bipartite cover number of a difference graph?
  \item Is there a sequence of graphs with constant local difference
    graph cover number and unbounded local complete bipartite cover number?
  \end{enumerate}
  We add to the study of these local covering numbers with a lower
  bound construction and some examples. Following Kim \emph{et al.}, we
  use the results on local covering numbers to provide lower and upper
  bounds for the local dimension of partially ordered sets of
  height~2. We discuss the local dimension of some posets related
  to Boolean lattices and show that the poset induced by
  the first two layers of the Boolean lattice has local dimension
  $(1 + o(1))\log_2\log_2 n$.  We conclude with some remarks on covering
  numbers for digraphs and Ferrers dimension.
\end{abstract}

\section{Introduction}

The \textit{covering number} of a graph $H$ (host) with respect to a class
\Guestclass\ is the least~$k$ such that there are graphs
$G_1,\ldots,G_k \in\Guestclass$ with $G_i \subset H$ for $i\in[k]$
such that their union covers the edges of $H$ and no other edges. We denote this number
by $\cn{}{\Guestclass}{H}$. The study of covering numbers has a long
tradition:
\begin{itemize}
\item
  In 1891 Petersen~\cite{Pet91} showed that the covering number of $2k$-regular
  graphs with respect to $2$-regular graphs is $k$.
\item
  In 1964 Nash-William~\cite{NW64} defined the \textit{arboricity} of a graph as
  the covering number respect to forests and showed that it equals the
  lower bound given by the maximum local density.
\item The \emph{track number} introduced by Gy\'arf\'as and
  West~\cite{GW95} and the \emph{thickness} introduced by Aggarwal
  \emph{et al.}~\cite{AKL+85} correspond to covering numbers with
  respect to interval graphs and planar graphs respectively.
\end{itemize}
Knauer and Ueckerdt~\cite{KU16} proposed the study of \emph{local
  covering numbers}.  This number is defined as the minimum number $k$
such that there is a cover of~$H$ with graphs from \Guestclass\ (see
above) such that every vertex of $H$ is contained in at most $k$
members of the cover.  We denote the local covering number by
$\cn{\ell}{\Guestclass}{H}$.  Fishburn and Hammer~\cite{FH96}
introduced the \emph{bipartite degree} which equals what we call the
local covering number with respect to complete bipartite graphs.
Motivated by questions regarding the local dimension of posets, Kim et
al.~\cite{KMMSSUW18} studied local covering number with respect to
difference graphs and compare this to the local covering number with
respect to complete bipartite graphs.

In this paper we continue the studies initiated in~\cite{KMMSSUW18}.
In Section~\ref{sec:covering-nums} we discuss local coverings with
difference graphs and complete bipartite graphs.  With
Theorem~\ref{thm:main-simple} we give a precise result regarding the local
covering number of a difference graph (Young diagram) with respect to
complete bipartite graphs (generalized rectangles). This answers a
question raised in~\cite{KMMSSUW18}.

Section~\ref{sec:local-dimension} relates the results to the local dimension of
posets. In this section we also discuss aspects of the local dimension
of Boolean lattices. Finally, in Section~\ref{sec:ferrers} we
discuss covering numbers of directed graphs and their relation with
order dimension and notions of Ferrers dimension.

\section{Covering numbers}
\label{sec:covering-nums}

Following the notation in~\cite{KU16}, local covering numbers are
defined as follows.  For a graph class $\Guestclass$ and a graph
$\Host$, an \emph{$\Guestclass$-covering} of $\Host$ is a
set of graphs $\Guest_1,\ldots,\Guest_t \in \Guestclass$ with $\Host =
\Guest_1 \cup \cdots \cup \Guest_t$. (In \cite{KU16} this is called an injective $\Guestclass$-covering. But as all coverings considered here are injective, we omit this specification throughout.) An 
$\Guestclass$-covering of $\Host$ is \emph{$k$-local} if every vertex
of $\Host$ is contained in at most $k$ of the graphs
$\Guest_1,\ldots,\Guest_t$, and the \emph{local $\Guestclass$-covering
number} of $\Host$, denoted by $\cn{\ell}{\Guestclass}{\Host}$, is the
smallest $k$ for which a $k$-local $\Guestclass$-cover of
$\Host$ exists.

A \emph{difference graph} is a bipartite graph in which the vertices
of one partite set can be ordered $a_1,\ldots,a_r$ in such a way that
$N(a_i) \subseteq N(a_{i-1})$ for $i=2,\ldots,r$, i.e., the
neighbourhoods of these vertices along this ordering are weakly
nesting.

Difference graphs are closely related to Young diagrams.  Let
$\mathbb{N}$ denote the set of positive integers.  For
$x \in \mathbb{N}$ we denote $[x] = \{1,\ldots,x\}$.  A \emph{Young
  diagram} with $r$ rows and $c$ columns is a subset
$Y \subseteq [r] \times [c]$ such that whenever $(i,j) \in Y$, then
$(i-1,j) \in Y$ provided $i \geq 2$, as well as $(i,j-1) \in Y$
provided $j \geq 2$.  A Young diagram\footnote{In the literature our
  Young diagrams are more frequently called Ferrers diagrams. We stick
  to Young diagram to be consistent with \cite{KMMSSUW18}.}  is
visualized as a set of axis-aligned unit squares, called \emph{cells}
that are arranged consecutively in rows and columns, each row starting
in the first column, and with every row (except the first) being at
most as long as the row above.

A \emph{generalized rectangle} (also called combinatorial rectangle) in a Young diagram
$Y \subseteq [r] \times [c]$ is a set $R$ of the form $R = S \times T$
with $S \subseteq [r]$ and $T \subseteq [c]$ and $R\subseteq Y$. Note
that (unless $Y = [r]\times [c]$) not every set of the form
$R = S \times T$ with $S \subseteq [r]$ and $T \subseteq [c]$
satisfies $R \subseteq Y$.  A generalized rectangle $R = S \times T$
with $S$ being a set of consecutive numbers in $[r]$ and $T$ being a
set of consecutive numbers in $[c]$ is an \emph{actual rectangle}.  A
generalized rectangle $R = S \times T$ \emph{uses} the rows in~$S$ and
the columns in $T$.  See \cref{fig:Young-diagram} for an
illustrative example.
\medskip

Difference graphs can be characterized as those bipartite graphs
$H = (V,E)$ with bipartition $V = A \disver B$, $|A| = r, |B| = c$,
which admit a bipartite adjacency matrix
$M = (m_{s,t})_{s \in A, t \in B}$ whose support is a Young diagram
$Y \subseteq [r] \times [c]$:
\[ \forall s \in A, t \in B\colon \qquad \{s,t\} \in E \quad
\Leftrightarrow \quad (s,t) \in Y \quad \Leftrightarrow \quad m_{s,t}
= 1
\]
Moreover, a complete bipartite subgraph $G$ of $H$ corresponds
to a generalized rectangle $R$ in $Y$.  Rows and columns of $M$
correspond to vertices of $H$ in $A$ and $B$, respectively.

\begin{figure}[ht]
 \centering
 \includegraphics[scale=.5]{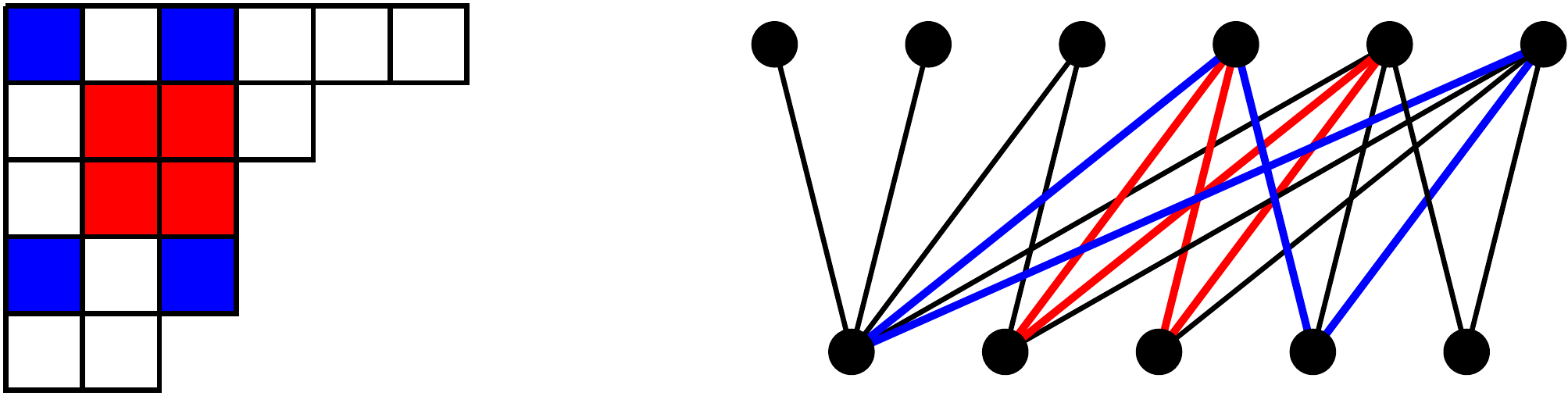}
 \caption{ \textbf{Left:} A Young diagram $Y$ with $r=5$ rows and
   $c=6$ columns.  Highlighted are the generalized rectangle
   $\{1,4\} \times \{1,3\}$ (blue), and the actual rectangle
   $\{2,3\} \times \{2,3\}$ (red).  \textbf{Right:} A corresponding
   difference graph with corresponding red and blue $K_{2,2}$.}
 \label{fig:YD+diffG}
\end{figure}

In~\cite{KMMSSUW18}, Kim \textit{et al.} introduced the concept of
covering a Young diagram with generalized rectangles subject to
minimizing the maximum number of rectangles intersecting any row or
column.  Their motivation was to investigate the relations between
\emph{local difference cover numbers} and \emph{local complete
  bipartite cover numbers}.

Let $\Differences$ denote the class of all difference graphs, and
$\CBipartites \subset \Differences$ the class of all complete
bipartite graphs.  Clearly, we have $\cn{\ell}{\Differences}{\Host}
\leq \cn{\ell}{\CBipartites}{\Host}$ for all graphs~$\Host$.  Kim
\textit{et al.}~\cite{KMMSSUW18} asked whether there is a sequence of
graphs $(\Host_i \colon i \in \mathbb{N})$ for which
$\cn{\ell}{\Differences}{\Host_i}$ is constant while
$\cn{\ell}{\CBipartites}{\Host_i}$ is unbounded.  They prove that for
all graphs $\Host$ on $n$ vertices,
\[ \cn{\ell}{\CBipartites}{\Host} \leq \cn{\ell}{\Differences}{\Host}
\cdot \left\lceil \log_2(n/2+1) \right\rceil,
\] by showing that $\cn{\ell}{\CBipartites}{\Host} \leq \lceil
\log_2(r+1) \rceil$ whenever $\Host \in \Differences$ is a difference
graph with one partite set of size $r$.  However, no lower bound on
$\cn{\ell}{\CBipartites}{\Host}$ for $\Host \in \Differences$ is
established in~\cite{KMMSSUW18}.  Specifically, Kim \textit{et al.}
ask for the exact value of $\cn{\ell}{\CBipartites}{\Host_i}$ for the
difference graph $H_i$ with vertex set $\{a_1,\ldots,a_i\} \cup
\{b_1,\ldots,b_i\}$ and $N(a_j) = \{b_1,\ldots,b_j\}$ for all $j\in [i]$.
For the case that $i+1$ is a power of~$2$ they prove the upper bound
$\cn{\ell}{\CBipartites}{\Host_i} \leq \log_2(i+1)-1$.

\medskip
The number of steps of a Young diagram $Y\subseteq [r] \times [c]$ is
the number of different row lengths in $Y$, i.e., the cardinality of
\[
  Z = \{ (s,t) \in Y \mid (s+1,t) \notin Y \text{ and } (s,t+1)
\notin Y\}.
\]
The cells in $Z$ are called the \emph{steps} of $Y$.  Young
diagrams with $n$ elements, $r$ rows, $c$ columns, and $z$ steps,
visualize partitions of $n$ into~$r$ unlabeled
summands (row lengths) with summands of $z$ different values and
largest summand being $c$.

We say that $Y$ is \emph{covered} by a set $C$ of generalized
rectangles if $Y = \bigcup_{R \in C} R$, i.e., $Y$ is the union of all
rectangles in $C$.  In this case we also say that $C$ is a
\emph{cover} of $Y$.  If additionally the rectangles in $C$ are
pairwise disjoint, we call $C$ a \emph{partition} of $Y$.  For
example, \cref{fig:Young-diagram} shows a Young diagram with a
partition into actual rectangles.

\begin{theorem}\label{thm:main-simple}
  For any $k \in \mathbb{N}$, any Young diagram $Y$ can be covered by
  a set $C$ of generalized rectangles such that each row and each
  column of $Y$ is used by at most $k$ rectangles in $C$ if and only if
  $Y$ has strictly less than $\binom{2k}{k}$ steps.
\end{theorem}

The Young diagram of the difference graph $H_i$ is
 $Y_i = \{ (s,t) \in [i] \times [i] \mid s
+ t \leq i+1 \}$, i.e., the (unique) Young diagram with $i$ rows, $i$
columns, and $i$ steps. Therefore \cref{thm:main-simple}
answers the questions raised by Kim~\textit{et al.}.

\subsection{Proof of \cref{thm:main-simple}}
\label{ssec:proof-main}

Throughout we shall simply use the term \emph{rectangle} for
generalized rectangles, and rely on the term \emph{actual rectangle}
when specifically meaning rectangles that are contiguous.  For a Young
diagram $Y$ and $i,j \in \mathbb{N}$, let us define a cover $C$ of $Y$
to be \emph{$(i,j)$-local} if each row of $Y$ is used by at most $i$
rectangles in $C$ and each column of $Y$ is used by at most $j$
rectangles in $C$.  Recall that $Y_z$ is the Young diagram with $z$
rows, $z$ columns, and $z$ steps. See \cref{fig:Young-diagram}.

\begin{figure}[ht]
 \centering
 \includegraphics{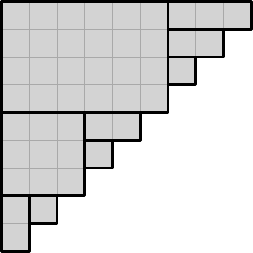}
 \caption{
The Young diagram
$Y_9$ with $9$ steps and a $(2,3)$-local partition of $Y$ with actual
rectangles.}
 \label{fig:Young-diagram}
\end{figure}

We start with a lemma stating that instead of considering any Young
diagram with $z$ steps, we may restrict our attention to just $Y_z$.

\begin{lemma}\label{lem:reduction-to-Yk}
  Let $i,j,z \in \mathbb{N}$ and $Y$ be any Young diagram with $z$
  steps.  Then~$Y$ admits an $(i,j)$-local cover if and only if $Y_z$
  admits an $(i,j)$-local cover with exactly $z$ rectangles.
\end{lemma}
\begin{proof}
First assume that $Y$ admits an $(i,j)$-local cover $C$.
If $C$ consists of strictly more than $z$ rectangles, as every rectangle is contained in a $[s]
\times [t]$ for some step $(s,t) \in Z$, by the pigeonhole principle there are
$R_1,R_2 \in C$, $R_1 \neq R_2$, such that $R_1,R_2 \subseteq [s]
\times [t]$ for some step $(s,t) \in Z$. However, in this case $C -
\{R_1,R_2\} + \{R_1 \cup R_2\}$ is also an $(i,j)$-local cover of $Y$
with one rectangle less, where $ \{R_1 \cup R_2\}$ denotes the rectangle whose row set and column set is the union of the row set and column set of $R_1$ and $R_2$. Thus, by repeating this argument, we may
assume that $|C| = z$.

 \begin{figure}[t] \centering \includegraphics{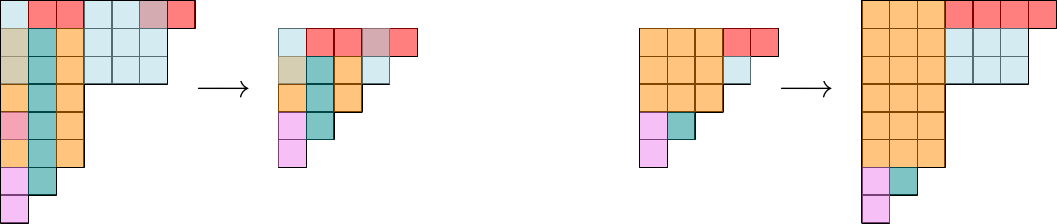}
  \caption{ Transforming a cover of any Young diagram $Y$ with $5$
steps into a cover of $Y_5$ (left) and vice versa (right).  }
  \label{fig:Yk-lemma}
 \end{figure}

If $Y \neq Y_z$, there is a row $s$ or a column $t$ that is not used
by any step in $Z$.
Apply the mapping $\mathbb{N} \times \mathbb{N} \to
\mathbb{N} \times \mathbb{N}$ with
\begin{eqnarray*}
(x,y) \mapsto
  \begin{cases} (x,y) & \text{ if } x < s\\
              (x-1,y) & \text{ if } x \geq s
  \end{cases}
& \text{ respectively } &
(x,y) \mapsto
  \begin{cases} (x,y) & \text{ if } y < t\\
              (x,y-1) & \text{ if } y \geq t
  \end{cases}
\end{eqnarray*}
Intuitively, we cut out row~$s$ (respectively column $t$), moving all
rows below one step up (respectively all columns to the right one step
left).  This gives an $(i,j)$-local cover of a smaller Young diagram
with $z$ steps, and eventually leads to an $(i,j)$-local cover of
$Y_z$, as desired.  See the left of \cref{fig:Yk-lemma}.  \medskip

On the other hand, if $Y_z$ admits an $(i,j)$-local cover
$C = \{R_1,\ldots,R_z\}$, this defines an $(i,j)$-local cover of $Y$
as follows.  Index the rows used by the steps $Z$ of $Y$ by
$s_1 < \cdots < s_z$ and the columns used by the steps $Z$ of $Y$ by
$t_1 < \cdots < t_z$ and let $s_0 = t_0 = 0$.  Defining
 \[ R'_a = \{(s,t) \in Y \mid s_{x-1} < s \leq s_x \text{ and }
t_{y-1} < t \leq t_y \text{ for some } (x,y) \in R_a\}
 \] for $a = 1,\ldots,z$ gives an $(i,j)$-local cover
$\{R'_1,\ldots,R'_z\}$ of $Y$.  See the right of \cref{fig:Yk-lemma}.

Observe that the construction maps an actual rectangle $R_a$ of $Y_z$
to an actual rectangle $R'_a$ of $Y$.  Also, if $\{R_1,\ldots,R_z\}$
is a partition of $Y_z$, then $\{R'_1,\ldots,R'_z\}$ is a partition of
$Y$.  This will be used in the proof of \cref{enum:construction} of
\cref{thm:main-general}.
\end{proof}

Let us now turn to our main result.
In fact, we shall prove the following strengthening of \cref{thm:main-simple}.

\begin{theorem}\label{thm:main-general}
For any $i,j,z \in \mathbb{N}$
and any Young diagram $Y$ with $z$ steps, the following hold.
 \begin{enumerate}[label = (\roman*)]
  \item If $z < \binom{i+j}{i}$, then there exists an $(i,j)$-local
    partition of $Y$ with actual rectangles.
    \label{enum:construction}
  \item If $z \geq \binom{i+j}{i}$, then there exists no $(i,j)$-local
    cover of $Y$ with generalized rectangles.
    \label{enum:argument}
 \end{enumerate}
\end{theorem}
\begin{proof}
 First, let us prove \cref{enum:construction}.
 For shorthand notation, we define $f(i,j) := \binom{i+j}{i} - 1$.
 It will be crucial for us that the numbers $\{f(i,j)\}_{i,j \geq 1}$ solve the recursion
\begin{align}
 f(i,j) = \begin{cases}
           f(i-1,j)+f(i,j-1)+1 & \text{ if } i,j \geq 2\\
           j & \text{ if } i=1, j \geq 1\\
           i & \text{ if } i \geq 1, j=1.
          \end{cases}\label{eq:recursion}
\end{align}
This follows directly from Pascal's rule
$\binom{a}{b} = \binom{a-1}{b-1} + \binom{a-1}{b}$ for any
$a,b \in \mathbb{N}$ with $1 \leq b \leq a-1$.

Due to \cref{lem:reduction-to-Yk} it suffices to show that for any
$i,j \in \mathbb{N}$ and $z = f(i,j) = \binom{i+j}{i}-1$, there is an
$(i,j)$-local partition of $Y_z$ with actual rectangles.

We define the $(i,j)$-local partition $C$ by induction on $i$ and $j$.
For illustrations refer to~\cref{fig:construction}.

\begin{figure}
 \centering
 \includegraphics{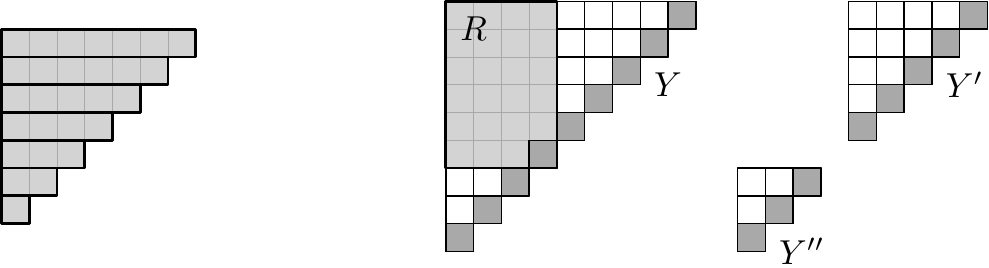}
 \caption{ \textbf{Left:} The Young diagram $Y_z$ with
   $z = f(1,7) = \binom{1+7}{1} - 1 = 7$ steps and a $(1,7)$-local
   partition of $Y_z$ into actual rectangles.  \textbf{Right:} The
   Young diagram $Y_z$ with $z = f(3,2) = \binom{3+2}{3} - 1 = 9$
   steps, the rectangle $R = [a] \times [z+1-a] = [6] \times [4]$ with
   $a = f(2,2)+1 = 6$, and the Young diagrams $Y'$ and $Y''$ with
   $f(2,2) = 5$ and $f(3,1) = 3$ steps, respectively.  }
 \label{fig:construction}
\end{figure}

If $i=1$, respectively $j=1$, then $C$ is the set of rows of $Y_j$,
respectively the set of columns of $Y_i$.  If $i \geq 2$ and
$j \geq 2$, then $z = f(i,j) = f(i-1,j) + f(i,j-1) + 1$
by~\eqref{eq:recursion}.  Consider the actual rectangle
$R = [a] \times [z+1-a]$ for $a = f(i-1,j) + 1$.  Then $Y_z - R$
splits into a right-shifted copy $Y'$ of $Y_{a-1}$ and a down-shifted
copy~$Y''$ of $Y_{z-a}$.  Note that $a-1 = f(i-1,j)$ and
$z-a = f(i,j-1)$.

By induction we have an $(i-1,j)$-local cover $C'$ of $Y'$
and an $(i,j-1)$-local cover $C''$ of $Y''$, each consisting
of pairwise disjoint actual rectangles.
Define
  \[
   C = \{R\} \cup C' \cup C'',
 \]
this is a cover of $Y_z$ consisting of pairwise disjoint actual
rectangles.  Rows $1$ to~$a$ are used by $R$ and at most $i-1$
rectangles in $C'$, and rows $a+1$ to $z$ are used by at most~$i$
rectangles in $C''$.  Hence each row of $Y_z$ is used by at most $i$
rectangles in $C$.  Similarly each column of $Y_z$ is used by at most
$j$ rectangles in~$C$.  Thus $C$ is an $(i,j)$-local partition of
$Y_z$ by actual rectangles, as desired.

For $z' < z = f(i,j)$ we obtain an $(i,j)$-local partition of $Y_{z'}$
by restricting the rectangles of the cover $C$ of $Y_z$ to the rows
from $z-z'$ to $z$. This yields an $(i,j)$-local partition of a
down-shifted copy $Y'$ of $Y_{z'}$.

\bigskip

Now, let us prove \cref{enum:argument}. Due to
\cref{lem:reduction-to-Yk} it is sufficient to show that for
$i,j \in \mathbb{N}$ the Young diagram $Y_{z'}$ with
$z' \geq \binom{i+j}{i}$ admits no $(i,j)$-local cover. If $Y_{z'}$
with $z' > z = \binom{i+j}{i}$ has an $(i,j)$-local cover, then by
restricting the rectangles of the cover to the rows from $z'-z$ to
$z'$ we obtain an $(i,j)$-local cover of a down-shifted copy of
$Y_z$. Therefore, we only have to consider $Y_z$.

Let $C$ be a cover of $Y_z$.  We shall prove that $C$ is not
$(i,j)$-local.  Again, we proceed by induction on $i$ and $j$, where
illustrations are given in~\cref{fig:argument}.

\begin{figure}
 \centering
 \includegraphics{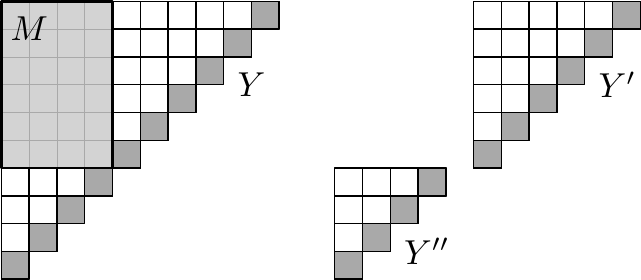}
 \caption{The Young diagram $Y_z$ with $z = \binom{3+2}{3} = 10$
   steps, the rectangle $M = [a] \times [z-a] = [6] \times [4]$ with
   $a = \binom{2+2}{2} = 6$, and the Young diagrams $Y'$ and $Y''$
   with $\binom{2+2}{2} = 6$ and $\binom{3+1}{3} = 4$ steps,
   respectively.}
 \label{fig:argument}
\end{figure}

If $i=1$, then each row is only used by a single rectangle in $C$,
otherwise,~$C$ would not be $(1,j)$-local. Hence, each row of $Y_z$ is a
rectangle in $C$. Thus column~$1$ of $Y_z$ is used by $z = j+1$
rectangles, proving that $C$ is not $(i,j)$-local.

The case $j=1$ is symmetric to the previous by exchanging rows and columns.

Now let $i \geq 2$ and $j \geq 2$.  We have
$z = \binom{i+j}{i} = \binom{(i-1)+j}{i-1} + \binom{i+(j-1)}{i}$.
Consider the rectangle $M = [a] \times [z-a]$ for
$a = \binom{(i-1)+j}{i-1}$.  Then $Y_z - M$
splits into a right-shifted $Y'$ copy of $Y_a$ and a
down-shifted copy $Y''$ of $Y_{z-a}$.
Note that $z-a = \binom{i+(j-1)}{i}$.

Let $C'$, respectively $C''$, be the subset of rectangles in $C$
using at least one of the rows $1,\ldots,a$ in $Y'$, respectively
at least one of the columns $1,\ldots,z-a$ in $Y''$. Note that
$C'\cap C'' = \emptyset$ as each generalized rectangle is contained
in~$Y_z$.

Prune each rectangle in $C'$ to the columns $z-a+1,\ldots,z$ and each
rectangle in~$C''$ to the rows $a+1,\ldots,z$. This yields
covers of $Y'$ and $Y''$.

The Young diagram $Y'$ is a copy of $Y_a$ and
$a=\binom{(i-1)+j}{i-1}$.  Hence, by induction the pruned cover $C'$
is not $(i-1,j)$-local.  If some column $t$ of $Y'$ is used by at
least $j+1$ rectangles in $C'$, this column of $Y_z$ is used by at
least $j+1$ rectangles in~$C$, proving that $C$ is not $(i,j)$-local,
as desired.  So we may assume that some row~$s$ of $Y'$ is used by at
least $i$ rectangles in $C'$.

Symmetrically, $Y''$ is a copy of $Y_{z-a}$ and $z-a = \binom{i+(j-1)}{i}$.
Hence, the pruned $C''$ is a cover of $Y''$, which by induction is not
$(i,j-1)$-local, and we may assume that some column $t$ of $Y''$ is
used by at least $j$ rectangles in $C''$.  Hence row $s$ in $Y_z$ is
used by at least $i$ rectangles in $C'$ and column $t$ in $Y_z$ is
used by at least $j$ rectangles in $C''$.  As
$C'\cap C'' = \emptyset$ and element $(s,t)$ is contained in
some rectangle of $C$, either row $s$ of $Y_z$ is used by at least
$i+1$ rectangles or column $t$ of $Y_z$ is used by at least $j+1$
rectangles (or both), proving that $C$ is not $(i,j)$-local.
\end{proof}

Finally, \cref{thm:main-simple} follows from \cref{thm:main-general}
by setting $i=j=k$.

\subsection{More about local covering numbers}
\label{ssec:more-covering}

Using \cref{thm:main-simple} and $\binom{2k}{k} = (1 +
o(1))\frac{1}{\sqrt{k\pi}}2^{2k}$, we see that
\begin{itemize}
 \item for every difference graph $\Host$ the exact value of
$\cn{\ell}{\CBipartites}{\Host}$ is the smallest $k \in \mathbb{N}$
such that for the number $z$ of steps\footnote{For graphs,
this is the number of different sizes of neighbourhoods in one partite
set.} of $\Host$ it holds that $z < \binom{2k}{k}$,

 \item the difference graphs $\Host_i$, $i \in \mathbb{N}$ (corresponding to the Young diagrams $Y_i$, $i \in \mathbb{N}$), defined by
Kim \textit{et al.} satisfy
  \[ \cn{\ell}{\CBipartites}{\Host_i} = \tfrac{1}{2}\log_2 i +
\tfrac{1}{4}\log_2\log_2 i + O(1),
  \] (and, using more precise bounds on Stirling's approximation, it
can be shown that the $O(1)$ term is at most $2$ for all $i\geq 2$),

 \item for this sequence $(\Host_i \colon i \in \mathbb{N})$ of
difference graphs $\cn{\ell}{\Differences}{\Host_i}$ is constant $1$,
while $\cn{\ell}{\CBipartites}{\Host_i}$ is unbounded, and

 \item for all graphs $\Host$ on $n$ vertices,
  \[ \cn{\ell}{\CBipartites}{\Host} \leq
\cn{\ell}{\Differences}{\Host} \cdot \left(\tfrac{1}{2}\log_2
\tfrac{n}{2}+\tfrac{1}{4}\log_2\log_2 \tfrac{n}{2} +2\right).
  \]
\end{itemize}

It is also interesting to understand the worst case scenario in covering
a bipartite graph by complete bipartite graphs or difference
graphs. With a different proof, the following result was already shown
in \cite{KMMSSUW18}.

\begin{theorem} For any $n$ there exists a bipartite graph $H$ on $n$
vertices such that
$\cn{\ell}{\Differences}{\Host}=\Omega(n/\log n)$.
\end{theorem}

\begin{proof}

\iftrue
Suppose $n\geq 2$ is even. Consider a random bipartite graph $G$ with vertex classes $A$ and $B$, where $|A| = |B| = n/2$ and each edge is chosen with probability $1/2$. For any $t \geq 2$, the expected number of $K_{t,t}$'s in $G$ is $2^{-t^2}\binom{n/2}{t}^2 < 2^{-t^2}\left(n/2\right)^{2t}$. If we choose $t = \lceil 2\log_2 n \rceil$, then the expected number of $K_{t,t}$'s (and hence the probability that $G$ contains a $K_{t,t}$) is less than $1/2$. The probability that $e(G) \geq \frac{1}{8}n^2$ is at least $1/2$, so with nonzero probability $e(G) \geq \frac{1}{8}n^2$ and $G$ has no $K_{t,t}$.

  Now consider a cover of $G$ with difference graphs.  We call a star
an $A$-star (resp. $B$-star) if its centre is in $A$
(resp. $B$). No difference graph in the cover contains a $K_{t,t}$ and thus every difference graph in the cover can be decomposed
into at most $t-1$ $A$-stars and at most $t-1$ $B$-stars.
Without loss of generality, at least half the edges of $G$ are covered by
$A$-stars. As $B$ has $n/2$ vertices, among the
$\frac{1}{16}n^2$ edges covered by $A$-stars there are at
least $n/8$ incident to some vertex $v\in B$.
Each difference graph in the cover contributes at most $t-1 \leq 2\log_2 n$ of
the $A$-stars containing $v$.  Therefore at least
$\frac{n}{16\log_2 n}$ difference graphs of the
cover contain $v$.
\fi
\end{proof}

As Kim et al.~\cite{KMMSSUW18} already observed, the upper bound follows from a theorem of Erd\H{o}s and
Pyber~\cite{EP97}, which shows that a cover of corresponding size
exists even with complete bipartite graphs.

\begin{theorem}[Erd\H{o}s and Pyber] For any simple graph $H$ on $n$
vertices, $\cn{\ell}{\CBipartites}{\Host}= O(n/\log n)$.
\end{theorem}

Hansel \cite{hansel} (see also Kir\'aly, Nagy, P\'alv\"{o}lgyi and
Visontai~\cite{KNPV12} and Bollob\'{a}s and
Scott~\cite{bollobas2007separating}) proved that
$\cn{\ell}{\CBipartites}{K_n}\geq \log_2 n$, which together with the
obvious upper bound gives the following proposition:

\begin{proposition}
\label{prop:hans} For all $n$, $\cn{\ell}{\CBipartites}{K_n} =
\lceil\log_2 n\rceil$.
\end{proposition}

In fact, Hansel proved a somewhat stronger result, namely that, for
every $\CBipartites$-covering of $K_n$, the average number
of complete bipartite graphs in which a vertex appears is at least
$\log_2 n$.

An interesting case is $K^\star_{n,n}$, which is obtained by deleting
the edges of a perfect matching from the complete bipartite graph
$K_{n,n}$. Note that $K^\star_{n,n}$ is the union of two difference
graphs. What is the best covering of this graph by complete bipartite
graphs? 

\begin{proposition}
\label{prop:kstar} $\cn{\ell}{\CBipartites}{K^\star_{n,n}}\le
\cn{\ell}{\CBipartites}{K_n}\le 2
\cn{\ell}{\CBipartites}{K^\star_{n,n}}$ and therefore
$\cn{\ell}{\CBipartites}{K^\star_{n,n}} = \Theta(\log n)$.
\end{proposition}

\begin{proof}
Let us denote the vertices of $K_n$ by $\{v_1, v_2,\dots v_n\}$, and
the vertices of $K^\star_{n,n}$ by $\{a_1, a_2, \dots a_n, b_1, b_2,
\dots b_n\}$ where $a_ib_j\in E(K^\star_{n,n})\Leftrightarrow
i\not=j$.  One can easily obtain a covering of $K^\star_{n,n}$ from a
covering of $K_n$. For every complete bipartite graph in the covering
with vertex classes $\{v_{i_1}, v_{i_2},\dots v_{i_p}\}$ and
$\{v_{j_1}, v_{j_2},\dots v_{j_q}\}$ take the complete bipartite graph
with vertex classes $\{a_{i_1}, a_{i_2},\dots a_{i_p}\}$ and
$\{b_{j_1}, b_{j_2},\dots b_{j_q}\}$ and another one with vertex
classes $\{b_{i_1}, b_{i_2},\dots b_{i_p}\}$ and $\{a_{j_1},
a_{j_2},\dots a_{j_q}\}$.  This will be a covering of $K^\star_{n,n}$
where the vertices $a_i$ and $b_i$ are covered exactly as many times
as $v_i$ in the covering of $K_n$. This construction shows
$\cn{\ell}{\CBipartites}{K^\star_{n,n}}\le\cn{\ell}{\CBipartites}{K_n}$.

On the other hand, we can obtain a covering of $K_n$ from a covering
of $K^\star_{n,n}$.  For a complete bipartite graph in the covering
with vertex classes $\{a_{i_1}, a_{i_2},\dots a_{i_p}\}$ and
$\{b_{j_1}, b_{j_2},\dots b_{j_q}\}$ take the complete bipartite graph
with vertex classes $\{v_{i_1}, v_{i_2},\dots v_{i_p}\}$ and
$\{v_{j_1}, v_{j_2},\dots v_{j_q}\}$.  This will be a covering of
$K_n$ where the vertex $v_i$ is covered exactly as many times as $a_i$
and $b_i$ are covered in total in the covering of
$K^\star_{n,n}$. This construction shows
$\cn{\ell}{\CBipartites}{K_n}\le 2
\cn{\ell}{\CBipartites}{K^\star_{n,n}}$.
\end{proof}
\section{Local dimension of posets}
\label{sec:local-dimension}

The motivation for Kim \textit{et al.}~\cite{KMMSSUW18} to study local
difference cover numbers comes from the local dimension of posets, a
notion recently introduced by Ueckerdt~\cite{U16}.

For a partially ordered set (also called a poset) $\mathcal{P} = (P,\leq)$,
define a \emph{realizer} as a set $\mathcal{L}$ of linear extensions
such that if $x$ and $y$ are incomparable (denoted $x || y$), then
$x < y$ in some $L \in \mathcal{L}$ and $y < x$ in some
$L' \in \mathcal{L}$. The \emph{dimension} of $\mathcal{P}$, denoted
$\dim(\mathcal{P})$, is the minimum size of a realizer. The dimension
of a poset is a widely studied parameter.

A \emph{partial linear extension} of $\mathcal{P}$ is a linear
extension $L$ of an induced subposet of $\mathcal{P}$.  A \emph{local
  realizer} of $\mathcal{P}$ is a non-empty set $\mathcal{L}$ of
partial linear extensions such that \textbf{(1)} if $x < y$ in
$\mathcal{P}$, then $x < y$ in some $L \in \mathcal{L}$, and
\textbf{(2)} if $x$ and $y$ are incomparable (denoted $x || y$), then
$x < y$ in some $L \in \mathcal{L}$ and $y < x$ in some
$L' \in \mathcal{L}$.  The \emph{local dimension} of $\mathcal{P}$,
denoted $\ldim(\mathcal{P})$, is then the smallest $k$ for which there
exists a local realizer $\mathcal{L}$ of $\mathcal{P}$ with each
$x \in P$ appearing in at most~$k$ partial linear extensions
$L \in \mathcal{L}$. Note that by definition
$\ldim(\mathcal{P}) \leq \dim(\mathcal{P})$ for every poset
$\mathcal{P}$.

For an arbitrary height-two poset $\mathcal{P} = (P,\leq)$, Kim
\textit{et al.} consider the bipartite graph $G_{\mathcal{P}} = (P,E)$
with partite sets $A = \min(\mathcal{P})$ (the minimal elements of $\mathcal P$) and $B = P -
\min(\mathcal{P}) \subseteq \max(\mathcal{P})$ whose edges correspond
to the so-called \emph{critical pairs}:
\[ \forall x \in A, y \in B \colon \qquad \{x,y\} \in E \quad
\Leftrightarrow \quad x || y \text{ in } \mathcal{P}
\] They prove that
\[ \cn{\ell}{\Differences}{G_{\mathcal{P}}} - 2 \leq
\ldim(\mathcal{P}) \leq \cn{\ell}{\CBipartites}{G_{\mathcal{P}}} + 2,
\] which also gives good bounds for $\ldim(\mathcal{P})$ when
$\mathcal{P}$ has larger height, since we have
\[ \ldim(\mathcal{Q})-2 \leq \ldim(\mathcal{P}) \leq
2\ldim(\mathcal{Q})-1
\] for the associated height-two poset $\mathcal{Q}$ known as the
split of $\mathcal{P}$ (see~\cite{BCPSTT17}, Lemma 5.5). Using these
results and the ones from the previous section, we can conclude the
following for the local dimension of any poset.

\begin{corollary} For any poset $\mathcal{P}$ on $n$ elements with
split $\mathcal{Q}$ we have
 \[ \cn{\ell}{\Differences}{G_{\mathcal{Q}}}-4 \leq \ldim(\mathcal{P})
\leq \cn{\ell}{\Differences}{G_{\mathcal{Q}}} \cdot (1 + o(1))\log_2
n.
 \]

\end{corollary}

\subsection{Local dimension of the Boolean lattice}
\label{ssec:boolean}

Let $2^{[n]}$ denote the Boolean lattice of subsets of the $n$ element set
$[n]$ (note that this lattice has $2^n$ elements, one for each subset
of $[n]$). Since the dimension of $2^{[n]}$ is $n$ we immediately have
$\ldim(2^{[n]})\le n$.

For any integer $s\in \{0,\dots,n\}$ let ${[n]}\choose{s}$ denote the
family of all the subsets of $[n]$ of size $s$.  This we call
\emph{layer} $s$ or the $s$'th layer of $2^{[n]}$ and let $P(s,t;n)$
be the subposet of $2^{[n]}$ induced by layers $s$ and $t$. We denote
$\ldim(s,t;n):=\ldim(P(s,t;n))$.

The study of the dimension $\dim(s,t;n):=\dim(P(s,t;n))$ has a long
history. Kierstead~\cite{Ki99} is a valuable survey on the topic.

Kim~\textit{et al.}~\cite{KMMSSUW18} give a lower bound for
$\ldim(1,n-\lceil n/e\rceil;n)= \Omega (n/\log n)$ which implies
that $\ldim(2^{[n]})= \Omega (n/\log n)$. For the height $2$ poset
$P(1,n-\lceil n/e\rceil;n)$ they in fact give bounds on the local
covering numbers of the corresponding bipartite graphs.

A similar lower bound on $\ldim(2^{[n]})$ can be obtained as follows: Let
$k = \ldim(2^{[n]})$ and consider a local realizer $L_1,\ldots,L_s$ such that
each subset of $[n]$ appears in at most~$k$ of the partial linear
extensions. Altogether there are at most $kn$ appearances of singletons and
at most $kn$ partial linear extensions containing a singleton. The singletons cut these partial linear extensions altogether into at most $2kn$ consecutive parts.  Given a non-singleton fixed set $A$ of $[n]$, for any given such part we have two options, $A$ is either present in this part or not. Moreover, $A$ is present in at most $k$ partial linear extensions, thus in at most $k$ such parts. Two sets $A$ cannot be present in exactly the same parts by the definition of a local realizer. Thus, the number of sets $A$ (which is equal to $2^n-n$) is at most the number of subsets of size at most $k$ on $2kn$ elements. Hence $2^n-n \leq \sum_{j=1}^k \binom{2nk}{j} \leq k\binom{2nk}{k}$. From the inequality it follows that $k \geq (1 - o(1))n/\log_2 n$.

\begin{problem}
  Determine the asymptotics of $\ldim(2^{[n]})$.
\end{problem}

A possible approach towards resolving the problem would be to study the
local dimension of appropriate pairs of levels.

We continue with what we can say regarding $\ldim(s,t;n)$
for some specific values of $s$ and $t$.

\subsection{The subposet of the two middle levels}

Let $n=2k+1$ be odd and consider the poset $P(k,k+1;2k+1)$ induced by the two
middle levels of the Boolean lattice. We are interested in
$\ldim(k,k+1;2k+1)$. More specifically, let $A(s,t;n)$ be the adjacency matrix
of the bipartite graph $G_{P(s,t;n)}$ defined by the critical pairs (see
Section~\ref{sec:local-dimension}) of levels $s$ and $t$. We want to find
good local covers of $M_{2k+1}= A(k,k+1;2k+1)$.

First we give a recursive formula for $M_{2k+1}$ (notice that it has
${2k+1} \choose {k}$ rows and colums). $M_1$ is the matrix with a
single $0$ entry (its single row corresponds to $\emptyset$, its
single column to $[1]$ and these are not connected in the
corresponding bipartite graph as $\emptyset\subset [1]$, that is, they
do not form a critical pair. Then,

$$M_{2k+1}=
\begin{bmatrix}
    M_{2k-1} &  \textbf{J} &    \overline{I} & \textbf{J}\\
  \textbf{J} &    M_{2k-1} &    \overline{I} & \textbf{J}\\
  \textbf{J} &  \textbf{J} & A(k-2,k-1;2k-1) & \textbf{J}\\
\overline{I} &\overline{I} &      \textbf{J} & A(k,k+1;2k-1)
\end{bmatrix}.
$$
Here \textbf{J} denotes an all-$1$ (not necessarily square) matrix of
appropriate size and $\overline{I}$ is the complement of an identity
(square) matrix of appropriate size. This recursion can be easily
verified by considering two elements $x,y\in [2k+1]$ and
ordering the rows and columns by first taking the sets
containing~$x$ but not containing $y$, then taking the sets
containing $y$ but not containing~$x$, then the sets containing
$x$ and $y$ and finally the sets containing none of the two.

\begin{problem}
  Determine the best local cover of $M_{2k+1}$ by Young diagrams.
\end{problem}

\subsection{The subposet of the first two levels}

In this subsection, we look at the poset $P(1,2;n)$ and the graph
$G_{P(1,2;n)}$ of critical pairs in $P(1,2;n)$. Throughout this section, we
identify the first layer with $[n]$ in the obvious way.

It is known that the dimension of $P(1,2;n)$ grows asymptotically as
$\log_2\log_2 n + (\frac{1}{2}+o(1))\log_2\log_2\log_2 n$. Spencer proved the
upper bound in~\cite{spencer}, and F\"{u}redi, Hajnal, R\"{o}dl, and Trotter
proved the corresponding lower bound in~\cite{fhrt}. The maximum $n$ such that
$\dim(1,2;n) \leq k$ is sometimes denoted $HM(k)$, see OEIS\footnote{On-Line
  Encyclopedia of Integer Sequences; {\tt https://oeis.org}} Sequence A001206
and Ho\c{s}ten and Morris~\cite{HS99}.

\begin{theorem} As $n\rightarrow\infty$, $\ldim(P(1,2;n)) =
\log_2\log_2 n + O(\log_2\log_2\log_2 n)$.
\end{theorem}
\begin{proof} The upper bound follows from Spencer's upper bound for
$\dim(P(1,2;n))$. We prove the lower bound
\begin{align*}
\ldim(P(1,2;n)) \geq
\cn{\ell}{\Differences}{G_{P(1,2;n)}} \geq \\
\log_2\log_2 n - \big(1 + \tfrac{1}{\ln 2}\big)\log_2\log_2\log_2 n - o(1).
\end{align*}

Let $\mathcal{D}$ be an $\Differences$-covering of
$G=G_{P(1,2;n)}$. Recall that, for each $D\in\mathcal{D}$, the
singletons in $D$ are weakly ordered by reverse inclusion of their
neighbourhoods. We define a sequence of difference graphs
$D_i\in\mathcal{D}$ and a sequence of subsets $L_i\subseteq[n]$ as
follows. Let $c<1$ be a fixed positive real number. First, choose
$D_1\in\mathcal{D}$ such that $D_1$ contains at least $n^c$
singeletons, if there is such a graph in $\mathcal{D}$. If there
isn't, then each pair is contained in at least $\frac{n-2}{n^c} \gg
\log_2\log_2 n$ elements of $\mathcal{D}$. Otherwise, let $L_1$ be the
set of singletons in $D_1$. Now suppose $L_i$ and $D_1,D_2,\dots,D_i$
have already been chosen. We choose a graph $D_{i+1}\in\mathcal{D}$
such that $V(D_{i+1})\cap L_i \geq |L_i|^c$, if such a graph
exists. If so, then, by the Erd\H{o}s-Szekeres theorem, there is a
subset $L_{i+1}\subseteq V(D_{i+1})\cap L_i$ such that $|L_{i+1}|\geq
|L_i|^{c/2}$ and the elements of $L_{i+1}$ appear in the same or
opposite order in $D_i$ and $D_{i+1}$. Continue in this way until
either $|L_i| \leq \log_2 n$ or $|L_i| > \log_2 n$ and there is no graph in $\mathcal{D}$
that contains $|L_i|^c$ elements of $L_i$. In the former case, each element
of $L_i$ appears in at least $i$ elements of $\mathcal{D}$, and
$n^{(c/2)^i}\leq\log_2 n$, so $i\geq \frac{1}{1-\log_2
c}(\log_2\log_2 n-\log_2\log_2\log_2 n)$. In the latter case, let $a$ and $b$
be the first and last elements of $L_i$ in the order induced by $D_i$
and look at the set of chosen difference graphs $D_j$ that contain the
pair $ab$. Because the ordering on $L_i$ induced by $D_j$ begins with
either $a$ or $b$ for every $j$, none of these graphs can contain any
edges from $ab$ to $L_i$. Every other difference graph in
$\mathcal{D}$ contains less than $|L_i|^c$ edges from $L_i$ to $ab$,
so there must be at least $\frac{|L_i|-2}{|L_i|^c} \geq (\log_2
n)^{1-c} - 2(\log_2 n)^{-c}$ such difference graphs containing
$ab$. Now, if we take $c = 1 - \frac{\log_2\log_2\log_2
n}{\log_2\log_2 n}$, then
\begin{align*}
(\log_2 n)^{1-c}-2(\log_2 n)^{-c} = \\
(\log_2 n)^{\frac{\log_2\log_2\log_2 n}{\log_2\log_2 n}} - 2(\log_2 n)^{-1+o(1)} = \\
2^{\cancel{\log_2\log_2 n} \cdot \frac{\log_2\log_2\log_2 n}{\cancel{\log_2\log_2 n}}} - o(1) = \\
\log_2\log_2 n - o(1).
\end{align*}
Using the affine approximation
\begin{align*}
\tfrac{1}{1-\log_2 c} = 1 + \tfrac{1}{\ln 2}(c-1) + O\left((c-1)^2\right)
\end{align*}
 as $c\to 1$, we have
\begin{align*}
\tfrac{1}{1-\log_2 c}(\log_2\log_2 n - \log_2\log_2\log_2 n) = \\
\left(1 - \tfrac{1}{\ln 2}\cdot \tfrac{\log_2\log_2\log_2 n}{\log_2\log_2 n} +
O\left((\tfrac{\log_2\log_2\log_2 n}{\log_2\log_2 n})^2\right)\right)(\log_2\log_2 n - \log_2\log_2\log_2 n) = \\
\log_2\log_2 n - \left(1 + \tfrac{1}{\ln 2}\right)\log_2\log_2\log_2 n - o(1)
.
\end{align*}
Therefore,
\begin{align*}
\cn{\ell}{\Differences}{G_{P(1,2;n)}} \geq
\min\big\{\log_2\log_2 n - o(1), \\ \log_2\log_2 n - \left(1 + \tfrac{1}{\ln 2}\right)\log_2\log_2\log_2 n - o(1)\big\},
\end{align*}
and the stated lower bound follows immediately.
\end{proof}

\begin{theorem} As $n\rightarrow\infty$,
$\cn{\ell}{\CBipartites}{G_{P(1,2;n)}} = \Theta(\log n)$.
\end{theorem}

\begin{proof} First we prove the lower bound. Choose any $x\in[n]$ and
consider the subgraph $F$ of $G_{P(1,2;n)}$ induced by the set of
singletons other than $x$ and the set of pairs containing $x$. $F$ is
a complete bipartite graph minus a matching, and the homomorphism
$\varphi : F\rightarrow K_{n-1}$ defined by $\varphi(y) = \varphi(xy)
= y$ is a double covering map. Hence if $B$ is an $\CBipartites$-covering of $G_{P(1,2;n)}$ and $B'$ is its restriction
to $F$, then $B'$ is an $\CBipartites$-covering of
$K^\star_{n-1,n-1}$. Therefore, by \cref{prop:hans} and \cref{prop:kstar},
$\cn{\ell}{\CBipartites}{G_{P(1,2;n)}} \geq \frac{1}{2}\log_2(n-1)$.

Now we prove the upper bound. Choose a random partition $A\cupdot B$
of $[n]$ and consider the complete bipartite subgraph of
$G_{P(1,2;n)}$ induced by the set of singletons in $A$ and the set of
pairs of elements of $B$. The edge $\{a,bc\}$ is covered by this
subgraph if and only if $a\in A$, $b,c\in B$, so the probability that
the edge is not covered is $\frac{7}{8}$. If we choose $3\log_{8/7} n$
such partitions independently, then the expected number of edges not
covered is $3\binom{n}{3}\big(\frac{7}{8}\big)^{3\log_{8/7} n} <
n^3\cdot n^{-3} = 1$. Therefore,
$\cn{\ell}{\CBipartites}{G_{P(1,2;n)}} \leq 3\log_{8/7} n$.
\end{proof}

\section{Ferrers Dimension}
\label{sec:ferrers}
\def\PP{\mathcal{P}}

Covering numbers and local covering numbers can also be defined for
directed graphs. In this section we provide some links to
research in this direction with emphasis to
questions regarding notions of dimension.

Recall that Young diagrams are more commonly called \emph{Ferrers diagrams}.
Riguet~\cite{Ri51} defined a \emph{Ferrers relation}\footnote{
  According to~\cite{EFO08} Ferrers relations have also been studied
  under the names of biorders, Guttman scales, and bi-quasi-series.
} 
as a relation  $R \subset X \times Y$ on possibly overlapping base sets $X$ and $Y$ such that
\begin{quote}
  $(x,y) \in R$ and $(x',y') \in R$ $\quad\implies\quad$
  $(x,y') \in R$ or $(x',y) \in R$.
\end{quote}
A relation $R \subset X \times Y$ can be viewed as a digraph $D$ with
$V_D = X\cup Y$ and $E_D=R$. A digraph thus corresponding to a Ferrers
relation is a \emph{Ferrers digraph}. Riguet characterized Ferrers
digraphs as those in which the sets $N^+(v)$ of out-neighbors are linearly
ordered by inclusion. Hence, bipartite Ferrers digraphs (i.e., when $X\cap Y=\emptyset$) are exactly
the difference graphs.

By playing with $x=x'$ and/or $y=y'$ in the definition of a Ferrers
relation it can be shown that Ferrers digraphs without loops are
\textbf{2+2}-free and transitive, i.e., they are interval orders.
In general, however, Ferrers digraphs may have loops.

In the spirit of order dimension the \emph{Ferrers dimension of a
  digraph} $D$ ($\fdim(D)$) is the minimum cardinality of a set of Ferrers digraphs
whose intersection is~$D$. If $\PP=(P,\leq)$ is a poset and $D_\PP$ the
digraph associated with the order relation (reflexivity implies that
$D_\PP$ has loops at all vertices), then $\dim(\PP) =
\fdim(D_\PP)$. This was shown by Bouchet~\cite{Bo71} and
Cogis~\cite{Co82}. The result implies that Ferrers dimension is a
generalization of order dimension. Since Ferrers digraphs are
characterized by having a staircase shaped adjacency matrix the
complement of a Ferrers digraph is again a Ferrers digraph.
Therefore, instead of representing a digraph as intersection of
Ferrers digraphs containing it ($D = \bigcap F_i$ with $D\subseteq F_i$),
we can as well represent its complement as union of Ferrers digraphs
contained in it ($\ovl{D} = \bigcup \ovl{F_i}$ with
$\ovl{F_i}\subseteq \ovl{D}$). This simple observation is sometimes
useful and indicates the connection to covering numbers, c.f.,
\cref{sec:covering-nums}.

The \emph{Ferrers dimension of a relation} $R$ ($\fdim(R)$) is
the minimum cardinality of a set of Ferrers relations whose intersection is $R$.
Note that if $D$ is the digraph corresponding to a relation $R$, then
$\fdim(D) = \fdim(R)$. Hence, the
result of Bouchet can be expressed as $\dim(\PP) =
\fdim(P,P,\leq)$, where we use the notation $(P,P,\leq)$ to
emphasize that we interpret the order as a relation.
The interval dimension $\operatorname{idim}(\PP)$ of a poset~$\PP$ is the
minimum cardinality of a set of interval orders extending $\PP$ whose intersection
is~$\PP$.  Interestingly, interval dimension is
also nicely expressed as a special case of Ferrers dimension:
$\operatorname{idim}(\PP) = \fdim(P,P,<)$.  For this and far reaching
generalizations see Mitas~\cite{Mi95}.

Relations $R \subset X \times Y$ with $X \cap Y=\emptyset$ can be
viewed as bipartite graphs. In this setting
$\fdim(R)$ is the global $\Differences$-covering number of
$\ovl{R}$, i.e., the minimum cardinality of a set of difference graphs whose union is
the bipartite complement of $R$.

We believe that it is worthwhile to study local variants of Ferrers
dimension.

\section*{Acknowledgments}

This paper has been assembled from drafts of three groups of authors
who had independently obtained Theorem~\ref{thm:main-simple}. The
research of Felsner and Ueckerdt has partly been published
in~\cite{FU19}. Most of their research was conducted during the
Graph Drawing Symposium 2018 in Barcelona. Thanks to Peter
Stumpf for helpful comments and discussions. Dam\'asdi, Keszegh and Nagy thanks the organizers of the 8th Emléktábla Workshop, where they started to work on these problems. They also thank Russ Woodroofe for his comments about the proof of Theorem \ref{thm:main-simple}.

\bibliography{lit}

\begin{thebibliography}{10}

\bibitem{AKL+85}
Alok Aggarwal, Maria~M. Klawe, David Lichtenstein, Nathan Linial, and Avi
  Wigderson.
\newblock Multi-layer grid embeddings.
\newblock In {\em Proc. FOCS}, pages 186--196, 1985.

\bibitem{BCPSTT17}
Fidel Barrera-Cruz, Thomas Prag, Heather~C. Smith, Libby Taylor, and William~T.
  Trotter.
\newblock {Comparing Dushnik-Miller Dimension, Boolean Dimension and Local
  Dimension}.
\newblock {\em arXiv preprint 1710.09467}, 2017.

\bibitem{bollobas2007separating}
B{\'e}la Bollob{\'a}s and Alex Scott.
\newblock On separating systems.
\newblock {\em European Journal of Combinatorics}, 28(4):1068--1071, 2007.

\bibitem{Bo71}
Andr\'e Bouchet.
\newblock {\em Etude combinatoire des ensembles ordonn\'es finis}.
\newblock These de Doctorat D'Etat, Universite de Grenoble, 1971.

\bibitem{Co82}
Olivier Cogis.
\newblock On the {F}errers dimension of a digraph.
\newblock {\em Discrete Math.}, 38:47--52, 1982.

\bibitem{EFO08}
David Eppstein, Jean-Claude Falmagne, and Sergei Ovchinnikov.
\newblock {\em Media theory}.
\newblock Springer, 2008.

\bibitem{EP97}
Paul Erd\H{o}s and L\'aszl\'o Pyber.
\newblock Covering a graph by complete bipartite graphs.
\newblock {\em Discrete Math.}, 170:249--251, 1997.

\bibitem{FU19}
Stefan Felsner and Torsten Ueckerdt.
\newblock A note on covering {Y}oung diagrams with applications to the local
  dimension of posets.
\newblock In {\em Proc. Eurocomb 2019}, volume~88 of {\em Acta Math. Univ.
  Comenianae}, pages 673--678, 2019.

\bibitem{FH96}
Peter~C. Fishburn and Peter~L. Hammer.
\newblock Bipartite dimensions and bipartite degrees of graphs.
\newblock {\em Discrete Math.}, 160:127--148, 1996.

\bibitem{fhrt}
Zoltan F\"{u}redi, P{\'e}ter Hajnal, Vojtech R\"{o}dl, and William~T. Trotter.
\newblock Interval orders and shift graphs.
\newblock In {\em Proc. Sets, graphs and numbers}, volume~60 of {\em Colloq.
  Math. Soc. J\'{a}nos Bolyai}, pages 297--313. North-Holland, 1992.

\bibitem{GW95}
Andr\'{a}s Gy\'{a}rf\'{a}s and Douglas West.
\newblock Multitrack interval graphs.
\newblock In {\em Proc. 26. {S}outheastern ICGTC}, volume 109 of {\em Congr.
  Numer.}, pages 109--116, 1995.
\newblock {\tt www.math.illinois.edu/~dwest/pubs/tracks.ps}.

\bibitem{hansel}
Georges Hansel.
\newblock Nombre minimal de contacts de fermeture n\'ecessaires pour r\'ealiser
  une fonction bool\'eenne sym\'etrique de $n$ variables.
\newblock {\em C. R. Acad. Sci. Paris}, 258, 1964.

\bibitem{HS99}
Serkan Ho\c{s}ten and Walter~D. Morris.
\newblock The order dimension of the complete graph.
\newblock {\em Discrete Math.}, 201:133--139, 1999.

\bibitem{Ki99}
Henry~A. Kierstead.
\newblock The dimension of two levels of the {B}oolean lattice.
\newblock {\em Discrete Math.}, 201:141--155, 1999.

\bibitem{KMMSSUW18}
Jinha Kim, Ryan~R. Martin, Tom{\'a}{\v{s}} Masa{\v{r}}{\'\i}k, Warren Shull,
  Heather~C. Smith, Andrew Uzzell, and Zhiyu Wang.
\newblock On difference graphs and the local dimension of posets.
\newblock {\em arXiv preprint 1803.08641}, 2018.
\newblock To appear {\it Europ. J. Comb.}

\bibitem{KNPV12}
Zolt\'an Kir\'aly, Zolt\'an~L. Nagy, D\"{o}m\"{o}t\"{o}r P\'alv\"{o}lgyi, and
  Mirk\'o Visontai.
\newblock On families of weakly cross-intersecting set-pairs.
\newblock {\em Fundamenta Informaticae}, 117:189--198, 2012.

\bibitem{KU16}
Kolja Knauer and Torsten Ueckerdt.
\newblock Three ways to cover a graph.
\newblock {\em Discrete Math.}, 339(2):745--758, 2016.

\bibitem{Mi95}
Jutta Mitas.
\newblock Interval orders based on arbitrary ordered sets.
\newblock {\em Discrete Math.}, 144:75--95, 1995.

\bibitem{NW64}
Crispin St. J.~A. Nash-Williams.
\newblock Decomposition of finite graphs into forests.
\newblock {\em J. London Math. Soc.}, 39:12, 1964.

\bibitem{Pet91}
Julius Petersen.
\newblock Die {T}heorie der regul\"{a}ren graphs.
\newblock {\em Acta Math.}, 15:193--220, 1891.

\bibitem{Ri51}
Jacques Riguet.
\newblock Les relations de {F}errers.
\newblock {\em {C. R. Acad. Sci., Paris}}, 232:1729--1730, 1951.

\bibitem{spencer}
Joel Spencer.
\newblock Minimal scrambling sets of simple orders.
\newblock {\em Acta Math. Acad. Sci. Hungar.}, 22:349--353, 1971.

\bibitem{U16}
Torsten Ueckerdt.
\newblock Order \& {G}eometry {W}orkshop, 2016.

\end{thebibliography}
\bibliographystyle{plain}

\end{document}